\documentclass[pdflatex,sn-mathphys-num]{sn-jnl}

\usepackage{graphicx}%
\usepackage{multirow}%
\usepackage{amsmath,amssymb,amsfonts}%
\usepackage{amsthm}%
\usepackage{mathrsfs}%
\usepackage[title]{appendix}%
\usepackage{xcolor}%
\usepackage{textcomp}%
\usepackage{manyfoot}%
\usepackage{booktabs}%
\usepackage{algorithm}%
\usepackage{algorithmicx}%
\usepackage{algpseudocode}%
\usepackage{listings}%

\theoremstyle{thmstyleone}%
\newtheorem{theorem}{Theorem}
\theoremstyle{thmstyletwo}%
\newtheorem{example}{Example}%
\newtheorem{lemma}{Lemma}%
\theoremstyle{thmstylethree}%
\newtheorem{definition}{Definition}%

\begin{document}

\title[Article Title]{The Impact of Data Dependence, Convergence and Stability by $AT$ Iterative Algorithms}


\author*[1,2]{\fnm{Akansha} \sur{ Tyagi}}\email{akanshatyagi0107@gmail.com ,atyagi1@maths.du.ac.in}

\author[2,3]{\fnm{Sachin} \sur{Vashistha}}\email{sachin.vashistha1@gmail.com}
\equalcont{These authors contributed equally to this work.}

\equalcont{These authors contributed equally to this work.}

\affil*[1]{\orgdiv{Department of Mathematics}, \orgname{University of Delhi}, \orgaddress{\street{Faculty of Mathematical Sciences Guru Tegh Bahadur}, \city{ Delhi}, \postcode{110007}, \state{ Delhi}, \country{India}}}

\affil[2]{\orgdiv{Department of Mathematics}, \orgname{Hindu College}, \orgaddress{\street{University of Delhi}, \city{Delhi}, \postcode{110007}, \state{Delhi}, \country{India}}}


\abstract{This article aims to present the $AT$ algorithm, a novel two-step iterative approach for approximating fixed points of weak contractions within complete normed linear spaces. The article demonstrates the convergence of  $AT$ algorithm towards fixed points of weak contractions. Notably, it establishes the algorithm's strong convergence properties, highlighting its faster convergence compared to established iterative methods such as $S$, normal-$S$, Varat, Mann, Ishikawa, $F^{*}$, and Picard algorithms. Additionally, the study explores the $AT$ algorithm's almost stable behavior for weak contractions. Emphasizing practical applicability, the paper offers data-dependent results through the $AT$ algorithm and substantiates findings with illustrative numerical examples}

\keywords{$AT$ iterative algorithm, Weak contraction, Fixed points, Numerically stable, Data dependence}



\maketitle
\section{Introduction and preliminaries}

In this paper, our foundational assumptions include considering $\mathbb{Z}_{+}$ as the collection of nonnegative integers. Additionally, we consider $P$ as a nonempty, closed and convex subset within a complete normed linear space $Q$. 
	Fixed point theory stands as a foundational and versatile framework in mathematics, offering powerful tools to study the existence and properties of solutions across diverse mathematical structures.Approximating the fixed points of both linear and nonlinear mappings through iterative methods stands as a cornerstone in fixed point theory,\\

	Beginning with foundational techniques like the Picard iteration \cite{bib1} which establishes convergence towards fixed points of contraction mappings to more intricate methodologies including Krasnoselskii\cite{bib2}, Mann\cite{bib3}, Ishikawa\cite{bib4},  S \cite{bib5}, normal-S \cite{bib6}, Varat\cite{bib7} , $F^{*}$\cite{bib8}  iterative algorithms constitute the backbone of computational approaches in fixed point theory for the self-mapping $R$ defined on $P$.\\

\begin{align}
	& \left\{
	\begin{array}{l}
		s_{0} \in P \\
		s_{m+1}= R s_{m}, \quad m \in \mathbb{Z}_{+}
	\end{array}
	\right. \label{eq:eq1} \\
	& \left\{
	\begin{array}{l}
		s_{0} \in P \\
		s_{m+1}=\left(1-a_{m}\right) s_{m}+a_{m} R s_{m}, \quad m \in \mathbb{Z}_{+}
	\end{array}
	\right. \label{eq:eq2} \\
	& \left\{
	\begin{array}{l}
		s_{0} \in P \\
		s_{m+1}=\left(1-a_{m}\right) s_{m}+a_{m} R b_{m} \\
		b_{m}=\left(1-d_{m}\right) s_{m}+d_{m} R s_{m}, \quad m \in \mathbb{Z}_{+}
	\end{array}
	\right. \label{eq:eq3} \\
	& \left\{
	\begin{array}{l}
		s_{0} \in P \\
		s_{m+1}=\left(1-a_{m}\right) R s_{m}+a_{m} R b_{m} \\
		b_{m}=\left(1-d_{m}\right) s_{m}+d_{m} R s_{m}, \quad m \in \mathbb{Z}_{+}
	\end{array}
	\right. \label{eq:eq4} \\
	& \left\{
	\begin{array}{l}
		s_{0} \in P \\
		s_{m+1}=R\left(\left(1-a_{m}\right) s_{m}+a_{m} R s_{m}\right), \quad m \in \mathbb{Z}_{+}
	\end{array}
	\right. \label{eq:eq5} \\
	& \left\{
	\begin{array}{l}
		s_{0} \in P \\
		s_{m+1}=\left(1-a_{m}\right) R t_{m}+a_{m} R b_{m} \\
		t_{m}=\left(1-c_{m}\right) s_{m}+c_{m} b_{m} \\
		b_{m}=\left(1-d_{m}\right) s_{m}+d_{m} R s_{m}, \quad m \in \mathbb{Z}_{+}
	\end{array}
	\right. \label{eq:eq6} \\
	& \left\{
	\begin{array}{l}
		s_{0} \in P \\
		s_{m+1}=R b_{m} \\
		b_{m}=R\left(\left(1-a_{m}\right) s_{m}+a_{m} R s_{m}\right), \quad m \in \mathbb{Z}_{+}
	\end{array}
	\right. \label{eq:eq7}
\end{align}

where  $a_{m}$, $c_{m}$, $d_{m}$  are sequences in  (0,1).\\
The iterations mentioned above \eqref{eq:eq1} , \eqref{eq:eq2}, \eqref{eq:eq3}, \eqref{eq:eq4}, \eqref{eq:eq5}, \eqref{eq:eq6}, \eqref{eq:eq7} have been proposed by distinguished Researchers.

Given  by above information, one  question arrive:\\
\textbf{Question} : Can we formulate a two-step iterative algorithm that converges  faster   by the $F^{*}$ iteration \eqref{eq:eq7} and other existing iterative methods ?\\
We introduce a novel iterative algorithm comprising two steps  called $AT$ algorithm, which is defined as follows:-\\
A complete normed linear space $Q$ has a nonempty closed and convex subset $P$. Given a self-mapping $R$ on $P$,
the sequence $\left\{s_{m}\right\}$ is defined by:

\begin{align}
	& \left\{
	\begin{aligned}
		& s_{0} \in P \\
		& s_{m+1} = R\left((1-a_m) b_m + a_m R b_m\right) \\
		& b_{m} = \frac{1}{2} R^2(s_m) + \frac{1}{2} R^2((1 - a_m) s_m + a_m R s_m), \quad m \in \mathbb{Z}_{+}
	\end{aligned}
	\right. \label{eq:eq8}
\end{align}

where $a_{m}$ is a sequence in $(0,1)$.\\

\begin{definition}\cite{bib16} A mapping $R: Q \rightarrow Q$ is considered to be $\zeta$-contraction with existing a constant $\zeta \in [0,1)$  such that:
\begin{equation}
\begin{aligned}
	\|R p - R q\| \leq \zeta \|p - q\|, \quad \forall p, q \in Q.\label{eq:eq9}
\end{aligned}
\end{equation}
\end{definition}

\begin{definition} \cite{bib9}  On a complete normed linear  space $Q$  a mapping $R: Q \rightarrow Q$ is considered to  as weak contraction with existing constant $\zeta \in(0,1)$ and some constant $L \geq 0$, such that:
\begin{equation}
	\|R p-R q\| \leq \zeta\|p-q\|+L\|q-R p\|, \quad \forall p, q \in Q 
	\label{eq:equation10}
\end{equation}
\end{definition}

\begin{theorem} \cite{bib9} On a complete normed linear  space $Q$  a mapping $R: Q \rightarrow Q$ with\\
condition \eqref{eq:equation10} and:

\begin{equation}
	\|R p - R q\| \leq \zeta \|p - q\| + L \|p - R p\|, \quad \forall p, q \in Q
	\label{eq:equation11}
\end{equation}

Consequently, there is a single fixed point for the mapping $R$ in $Q$.
\end{theorem}

\begin{definition} \cite{bib10}  Let  $s_{0} \in Q$ and $s_{m+1}=g\left(R, s_{m}\right)$ is defined an iterative method for a function $g$ on a complete normed linear space $Q$  with self-map $R$ having fixed point $s$. Let $\left\{r_{m}\right\}$ be sequence of an approximation of $\left\{s_{m}\right\}$ in $Q$ and define $\gamma_{m}=\| r_{m+1}-$ $g\left(R, r_{m}\right) \|$. Then iterative method $s_{m+1}=g\left(R, s_{m}\right)$ is known as $R$-stable if:

$$
\lim _{m \rightarrow \infty} \gamma_{m}=0 \Longleftrightarrow \lim _{m \rightarrow \infty} r_{m}=s
$$
\end{definition}

\begin{definition}  \cite{bib10} Let $s_{0} \in Q$ and $s_{m+1}=g\left(R, s_{m}\right), m \in \mathbb{Z}_{+}$is an iterative method for a function $g$  on a complete normed linear space $Q$  with self-map $R$ having fixed point $s$. Let $\left\{r_{m}\right\}$ be an sequence of approximate of $\left\{s_{m}\right\}$ in $Q$ and define $\gamma_{m}=\left\|r_{m+1}-g\left(R, r_{m}\right)\right\|$. Then iterative method $s_{m+1}=g\left(R, s_{m}\right)$ is known as almost $R$-stable if:

$$
\sum_{m=0}^{\infty} \gamma_{m}<\infty \Longrightarrow \lim _{m \rightarrow \infty} r_{m}=s
$$
\end{definition}
\begin{lemma}
\cite{bib11} Let $\left\{u_{m}\right\}$ and $\left\{v_{m}\right\}$ be two sequences in $\mathbb{R}_{+}$ and $0 \leq s<1$ so that $u_{m+1} \leq s u_{m}+v_{m} 
\forall m \geq 0$.

(i) If $\lim _{m \rightarrow \infty} v_{m}=0$ implies that $\lim _{m \rightarrow \infty} u_{m}=0$.
\end{lemma}

\begin{lemma} \cite{bib12} With existing $N \in \mathbb{Z}_{+}$ define a sequence $\left\{p_{m}\right\}$ in $\mathbb{R}_{+}$  so that for all $m \geq N$ satisfying the following inequality:

$$
p_{m+1} \leq \left(1-\delta_{m}\right) p_{m} + \delta_{m} q_{m}
$$

where $\delta_{m} \in (0,1)$ for all $m \in \mathbb{Z}_{+}$ such that $\sum_{m=0}^{\infty} \delta_{m} = \infty$ and $q_{m} \geq 0$. Then:

$$
0 \leq \lim_{m \rightarrow \infty} \sup p_{m} \leq \lim_{m \rightarrow \infty} \sup q_{m}.
$$
\end{lemma}
\begin{lemma} \cite{bib8} With existing $N \in \mathbb{Z}_{+}$ define a sequence $\left\{p_{m}\right\}$  in $\mathbb{R}_{+}$  in such a way  for all $m \geq N$, $\left\{p_{m}\right\}$ with the property:

$$
p_{m+1} \leq\left(1-\delta_{m}\right) p_{m}+\delta_{m} q_{m}
$$

where $\delta \in (0,1) \forall m \in \mathbb{Z}_{+}$, such that $\sum_{m=0}^{\infty} \delta_{m}=\infty$ and $\delta_{m} \geq 0$ which  define a  sequence whose terms are bounded . Then:

$$
0 \leq \lim _{m\rightarrow \infty} \sup p_{m} \leq \lim _{m \rightarrow \infty} \sup q_{m} .
$$
\end{lemma}

\begin{definition} Let $\left\{p_{m}\right\}$ and $\left\{q_{m}\right\}$ define two sequences which belongs to  $\mathbb{R}_{+}$ that converge to $p$ and $q$, respectively. Assume that:

$$
\ell=\lim _{n \rightarrow \infty} \frac{\left|p_{m}-p\right|}{\left|q_{m}-q\right|}
$$

(i) If $\ell=0$, then $\left\{p_{m}\right\}$ converges to $p$ faster than $\left\{q_{m}\right\}$ to $q$.

(ii) If $0<\ell<\infty$, then $\left\{p_{m}\right\}$ and $\left\{q_{m}\right\}$ both will have same convergence rate .
\end{definition}

\begin{definition} \cite{bib8} Let $\left\{\theta_{m}\right\}$
and $\left\{\eta_{m}\right\}$ be two iterative algorithms having the same point $\theta$ as point of convergence with the  error estimate:

$$
\begin{aligned}
	\left|\theta_{m}-\theta\right| & \leq p_{m} \\
	\left|\eta_{m}-\theta\right| & \leq q_{m}
\end{aligned}
$$
If $\lim_{m \rightarrow \infty} \frac{p_{m}}{q_{m}} = 0$, then convergence of $\{\theta_{m}\}$ is faster than $\eta_{m}$.

\end{definition}

\begin{definition} \cite{bib13} Let $F$ and $R$ be two self operators which is defined on a nonempty subset $P$ of a complete normed linear space $Q$. An operator $F$ is known as approximate operator of $R$ if existing a fixed $\epsilon>0$  such that $\|F p-R p\| \leq \epsilon$ for all $p \in P$.
	\end{definition}
\section{Main results}
By  $AT$ iteration in complete normed linear space, we will establish results which are related to weak contractions.\\
\begin {theorem} Let $R: P \rightarrow P$ is defined as a  weak contraction having the condition \ref{eq:equation11} , where $P$ state  a nonempty closed and convex subset of a complete normed linear space $Q$. Then, $\left\{s_{m}\right\}$ defined by $AT$ iterative algorithm \eqref{eq:eq8} reaches a single, exclusive fixed point of $P$.
\end{theorem}

\begin{proof}- By condition \eqref{eq:equation11}, we have:

$$
\begin{aligned}
	\left\|R s_{m}-s\right\| & =\left\|R s_{m}-R s\right\| \\
	& \leq \zeta \|s_{m}-s\|+L\|s-Rs\| \\
	& = \zeta \|s_{m}-s\|, \quad \forall m \in \mathbb{Z}_{+} .
\end{aligned}
$$

Using $AT$ iteration \eqref{eq:eq8} , we have:\\
$$
\begin{aligned}
	\| b_{m} - s \| &= \frac{1}{2} \| R^2(s_m) + R^2((1 - a_m) s_m + a_m Rs_m) - 2s \| \\
	&\leq \frac{1}{2} \| R^2(s_m) -s \| + \frac{1}{2} \| R^2((1 - a_m) s_m + a_m Rs_m) - s  \| \\
	& \leq \frac{\zeta}{2} \|Rs_m - s\| + \frac{\zeta}{2} \| R((1 - a_m) s_m + a_m R s_m) - s\|\\
	& \leq \frac{\zeta^2}{2} \|s_m - s\| + \frac{\zeta^2}{2} \|(1 - a_m) s_m + a_m R s_m - s\| \\
	&\leq \frac{\zeta^2}{2} \|s_m - s\| + \frac{\zeta^2}{2} \|(1 - a_m) s_m + a_m R s_m -(1 - a_m)s-a_ms\| \\
	&\leq \frac{\zeta^2}{2} \|s_m - s\| + (1 - a_m) \frac{\zeta^2}{2}\| s_m -s\|+\frac{\zeta^3}{2}a_m\|s_m - s\|\\
	&\leq \bigg(\frac{\zeta^2}{2}+(1 - a_m) \frac{\zeta^2}{2}+\frac{\zeta^2}{2}a_m \bigg)\|s_m - s\|\\
	&\leq\zeta^2\|s_m - s\| \\
\end{aligned} 
$$
so we have ,
\begin{equation}
	\| b_{m} - s \| \leq \zeta^2 \|s_m - s\|
	\label{eq:equation12}
\end{equation}

Using Eq.  \eqref{eq:equation12}, we get:
$$
\begin{aligned}
	\| s_{m+1} - s \|& = \| R((1 - a_m)b_m + a_m R (b_m)) - s\|\\
	&\leq \zeta\| ((1 - a_m)b_m + a_m R (b_m)) - s\|\\
	&\leq \zeta\bigg((1 - a_m)\| b_m -s\| + a_m\| R (b_m)) - s\|\bigg)\\
	&\leq \zeta\bigg((1 - a_m)\| b_m -s\| + a_m\zeta\| b_m - s\|\bigg)\\
	&\leq\zeta( (1 - a_m)+a_m\zeta)\| b_m - s\|\\
\end{aligned}
$$

As $0<\zeta<1$ and $a_m \in(0,1)$, therefore, using  $(1-(1-\zeta) a_{m}\leq 1$, we arrive the conclussion:

$$
\left\|s_{m+1}-s\right\| \leq \zeta^{3}\left\|s_{m}-s\right\| .
$$

Consequently, we get:

\begin{equation}
	\left\|s_{m+1}-s\right\| \leq \zeta^{3(m+1)}\left\|s_{0}-s\right\|
	\label{eq:equation13}
\end{equation}

Since $0<\zeta<1$, hence, $\left\{s_{m}\right\}$ converges strongly to $s$.\\
\vspace{0.1cm}

In the next theorem we will show the almost $R$-stability of $AT$ iterative algorithm .
\end{proof}
\begin{theorem} Let $R: P \rightarrow P$ be a mapping which defines to a weak contraction with the condition  \eqref{eq:equation11} where $P$ is a nonempty, closed, and convex subset of a complete normed linear  space $Q$. Then, $AT$ iterative algorithm \eqref{eq:eq8} is almost $R$-stable. 
	\end{theorem}

\begin{proof} Suppose that $\left\{r_{m}\right\}$ is an arbitrary sequence in $P$ and the sequence defined by $AT$ algorithm \eqref{eq:eq8} is $s_{m+1}=g\left(R, s_{m}\right)$ and $\gamma_{m}=\left\|r_{m+1}-g\left(R, r_{m}\right)\right\|, m \in \mathbb{Z}_{+}$. Now, we will show that:

$$
\sum_{m=0}^{\infty} \gamma_{m}<\infty \Longrightarrow \lim _{m \rightarrow \infty} r_{m}=s
$$

Let $\sum_{m=0}^{\infty} \gamma_{m}<\infty$, and then, by $AT$ algorithm \eqref{eq:eq8}, we have:

$$
\begin{aligned}
	\left\|r_{m+1}-s\right\| & \leq\left\|r_{m+1}-g\left(R, r_{m}\right)\right\|+\left\|g\left(R, r_{m}\right)-s\right\| \\
	& \leq \gamma_{m}+\zeta^{3}\left(1-(1-\zeta) {a_{m}}\right)\left\|r_{m}-s\right\| \\
\end{aligned}
$$

define $c_{m}=\left\|r_{m}-s\right\|$ and $v=\zeta^{3}(1-(1-\zeta)a_{m})$, and then, $0 \leq v<1$ and:

$$
c_{m+1} \leq v c_{m}+\gamma_{m}
$$

Thus, conclusion follows by Lemma 1.5.\\
\vspace{0.1cm}

The following theorem proves that $AT$ iterative algorithm converges faster than the algorithms  \eqref{eq:eq1} , \eqref{eq:eq2}, \eqref{eq:eq3}, \eqref{eq:eq4}, \eqref{eq:eq5}, \eqref{eq:eq6}, \eqref{eq:eq7} for weak contractions.
\end{proof}

\begin{theorem} Let $R: P \rightarrow P$ be a  mapping with weak contraction and \eqref{eq:equation11}, where $P$ is a nonempty, closed, and convex subset of a complete normed linear space $Q$. Let the sequences $\left\{s_{1, m}\right\},\left\{s_{2, m}\right\},\left\{s_{3, m}\right\}$, $\left\{s_{4, m}\right\},\left\{s_{5, m}\right\},\left\{s_{6, m}\right\}$, $\left\{s_{7, m}\right\}and \left\{s_{m}\right\}$ be defined by \eqref{eq:eq1} to \eqref{eq:eq8} , respectively,  converge to $s$ which is common  fixed point . Then,convergence of  $AT$ algorithm is faster towards fixed point  $s$   than \eqref{eq:eq1} to \eqref{eq:eq7}
.
\end{theorem}
\begin{proof}
	According to equation  \eqref{eq:equation13} from  Theorem 2.1,
	
	$$
	\left\|s_{m+1}-s\right\| \leq \zeta^{3(m+1)}\left\|s_{0}-s\right\|=\eta_{m}, m \in \mathbb{Z}_{+}
	$$
	
	by Khan \cite{bib14} Proposition 1 ,we have

	$$
	\left\|s_{1, m}-s\right\| \leq \zeta^{m+1}\left\|s_{1,0}-s\right\|=\eta_{1, m}, m \in \mathbb{Z}_{+}
	$$
	
	Then:
	
	$$
	\frac{\eta_{m}}{\eta_{1, m}}=\frac{\zeta^{3(m+1)}\left\|s_{0}-s\right\|}{\zeta^{m+1}\left\|s_{1,0}-s\right\|}=\zeta^{2(m+1)} \frac{\left\|s_{0}-s\right\|}{\left\|s_{1,0}-s\right\|}
	$$
	
	Since $0<\zeta<1$, therefore, we have $\frac{\eta_{m}}{\eta_{1, m}} \rightarrow 0$ as $m \rightarrow \infty$. Hence, the sequence $\left\{s_{m}\right\}$ converges faster than $\left\{s_{1, m}\right\}$ to $s$.
	
	Now, by normal-S algorithm \eqref{eq:eq5}, we have:
	
	$$
	\begin{aligned}
		\left\|s_{m+1}-s\right\| & =\left\|R\left(\left(1-a_{m})\right) s_{m}+{a_{m}} R s_{m}\right)-s\right\| \\
		& \leq \zeta\left[\left(1-a_{m})\right) s_{m}+a_{m}R s_{m}-s \|\right] \\
		&\leq \zeta\| ((1 - a_m)s_m + a_m R (s_m)) - s\|\\
		&\leq \zeta\bigg((1 - a_m)\| s_m -s\| + a_m\| R (s_m)) - s\|\bigg)\\
		&\leq \zeta\bigg((1 - a_m)\| s_m -s\| + a_m\zeta\| s_m - s\|\bigg)\\
		&\leq\zeta( (1 - a_m)+a_m\zeta)\| s_m - s\|\\
		& \leq \zeta\left\|s_{m}-s\right\| .
	\end{aligned}
	$$
	
	Similarly, we get:
	
	$$
	\left\|s_{m+1}-s\right\| \leq \zeta^{m+1}\left\|s_{0}-s\right\|\\
	$$
	Let,
	$$
	\left\|s_{5,m} - s\right\| \leq \zeta^{m+1}\left\|s_{5,0} - s\right\| = \eta_{5,m}
	$$
	Then:
	
	$$
	\frac{\eta_{m}}{\eta_{5, m}}=\frac{\zeta^{3(m+1)}\left\|s_{0}-s\right\|}{\zeta^{m+1}\left\|s_{5,0}-s\right\|}=\zeta^2{(m+1)} \frac{\left\|s_{0}-s\right\|}{\left\|s_{5,0}-s\right\|}
	$$
	
	We get $\frac{\eta_{m}}{\eta_{5, m}} \rightarrow 0$ as $m \rightarrow \infty$. Hence, the sequence $\left\{s_{m}\right\}$ converges faster than $\left\{s_{5, m}\right\}$ to the fixed point $s$.
	
	Sintunavarat and Pitea \cite{bib15} was proved
	
	$$
	\left\|s_{6, m}-s\right\| \leq \zeta^{m+1}[1-(1-\zeta) e(f-g+g f)]^{m+1}\left\|s_{6,0}-s\right\|=\eta_{6, m}, m \in \mathbb{Z}_{+} .
	$$
	
	And using t $1-(1-\zeta) e(f-g+g f) \leq 1$, we obtain:
	
	$$
	\left\|s_{6, m}-s\right\| \leq \zeta^{m+1}\left\|s_{6,0}-s\right\|=\zeta_{6, m}
	$$
	
	Then:
	
	$$
	\frac{\eta_{m}}{\eta_{6, m}}=\frac{\zeta^{3(m+1)}\left\|s_{0}-s\right\|}{\zeta^{m+1}\left\|s_{6,0}-s\right\|}=\zeta^{2(m+1)} \frac{\left\|s_{0}-s\right\|}{\left\|s_{6,0}-s\right\|}
	$$
	
	Thus, we get $\frac{\eta_{m}}{\eta_{6, m}} \rightarrow 0$ as $m \rightarrow \infty$. Hence, $\left\{s_{m}\right\}$ converges faster than $\left\{s_{6, m}\right\}$ to $s$.

	Also,Sintunavarat and Pitea \cite{bib15} ] given the result that the Varat algorithm converges faster than S, Mann and Ishikawa iterative algorithms for the mapping with weak contraction.
	showed  that the Varat algorithm converges faster than Mann, Ishikawa, and $\mathrm{S}$ iterative algorithms for the class of weak contractions. \\
	Now as for $F^{*}$ iteration we have,
	$$
	\begin{aligned}
		\left\|s_{m+1}-s\right\| & =\left\|R^2\left(\left(1-a_{m})\right) s_{m}+{a_{m}} R s_{m}\right)-s\right\| \\
		& \leq \zeta\left[R(\left(1-a_{m})\right) s_{m}+a_{m}R s_{m})-s \|\right] \\
		&\leq \zeta^2\| ((1 - a_m)s_m + a_m R (s_m)) - s\|\\
		&\leq \zeta^2\bigg((1 - a_m)\| s_m -s\| + a_m\| R (s_m)) - s\|\bigg)\\
		&\leq \zeta^2\bigg((1 - a_m)\| s_m -s\| + a_m\zeta\| s_m - s\|\bigg)\\
		&\leq\zeta^2( (1 - a_m)+a_m\zeta)\| s_m - s\|\\
		& \leq \zeta^2\left\|s_{m}-s\right\| .
	\end{aligned}
	$$
	Similarly, we get:
	
	$$
	\left\|s_{m+1}-s\right\| \leq \zeta^{2(m+1)}\left\|s_{0}-s\right\|\\
	$$
	Let,
	$$
	\left\|s_{7,m} - s\right\| \leq \zeta^{2(m+1)}\left\|s_{7,0} - s\right\| = \eta_{7,m}
	$$
	Then:
	
	$$
	\frac{\eta_{m}}{\eta_{7, m}}=\frac{\zeta^{3(m+1)}\left\|s_{0}-s\right\|}{\zeta^{2(m+1)}\left\|s_{7,0}-s\right\|}=\zeta^{(m+1)} \frac{\left\|s_{0}-s\right\|}{\left\|s_{7,0}-s\right\|}
	$$
	
	We get $\frac{\eta_{m}}{\eta_{7, m}} \rightarrow 0$ as $m \rightarrow \infty$. Hence, the sequence $\left\{s_{m}\right\}$ converges faster than $\left\{s_{7, m}\right\}$ to the fixed point $s$.

	Thus, $AT$ iterative algorithm converges faster than all the iterative algorithms \\
\end{proof}

\begin{example}
Let $P = \mathbb{R}$ be a complete normed linear space  norm define as usual norm and $P = [0,\pi]$, a subset of $Q$. Let $R : Q \rightarrow Q$ be a self-mapping defined by $Rx = \cos\left(\frac{x}{2}\right)$ for all $x \in P$. It can be easily verified that $R$ is a weak contraction satisfying , and $R$ has a unique fixed point $p = 0.9$. Choose the control sequences $a_{m} = 0.5$.
\end{example}

\begin{table}[]
	\centering
	\caption{Table 1 .A Comparative Analysis of Iterative Algorithms Applied to Example 1}
	\begin{tabular}{lllllll}
		iteration & AT       & $F^{*}$ & picard   & normal\_s & mann     & varat    \\
		0         & 1.658950 & 1.658950                 & 1.658950 & 1.658950  & 1.658950 & 1.658950 \\
		1         & 0.893291 & 0.934867                 & 0.675263 & 0.725825  & 1.517125 & 0.688976 \\
		2         & 0.900422 & 0.901728                 & 0.943542 & 0.929411  & 1.403039 & 0.939580 \\
		3         & 0.900367 & 0.900420                 & 0.890765 & 0.895096  & 1.310885 & 0.892081 \\
		4         & 0.900367 & 0.900369                 & 0.902446 & 0.901311  & 1.236186 & 0.902078 \\
		5         & 0.900367 & 0.900367                 & 0.899914 & 0.900198  & 1.175459 & 0.900012 \\
		6         & 0.900367 & 0.900367                 & 0.900466 & 0.900398  & 1.125969 & 0.900441 \\
		7         & 0.900367 & 0.900367                 & 0.900346 & 0.900362  & 1.085556 & 0.900352 \\
		8         & 0.900367 & 0.900367                 & 0.900372 & 0.900368  & 1.052500 & 0.900370 \\
		9         & 0.900367 & 0.900367                 & 0.900366 & 0.900367  & 1.025423 & 0.900367
	\end{tabular}
\end{table}

\begin{figure}[ht!]
	\centering
	\includegraphics[width=0.8\textwidth]{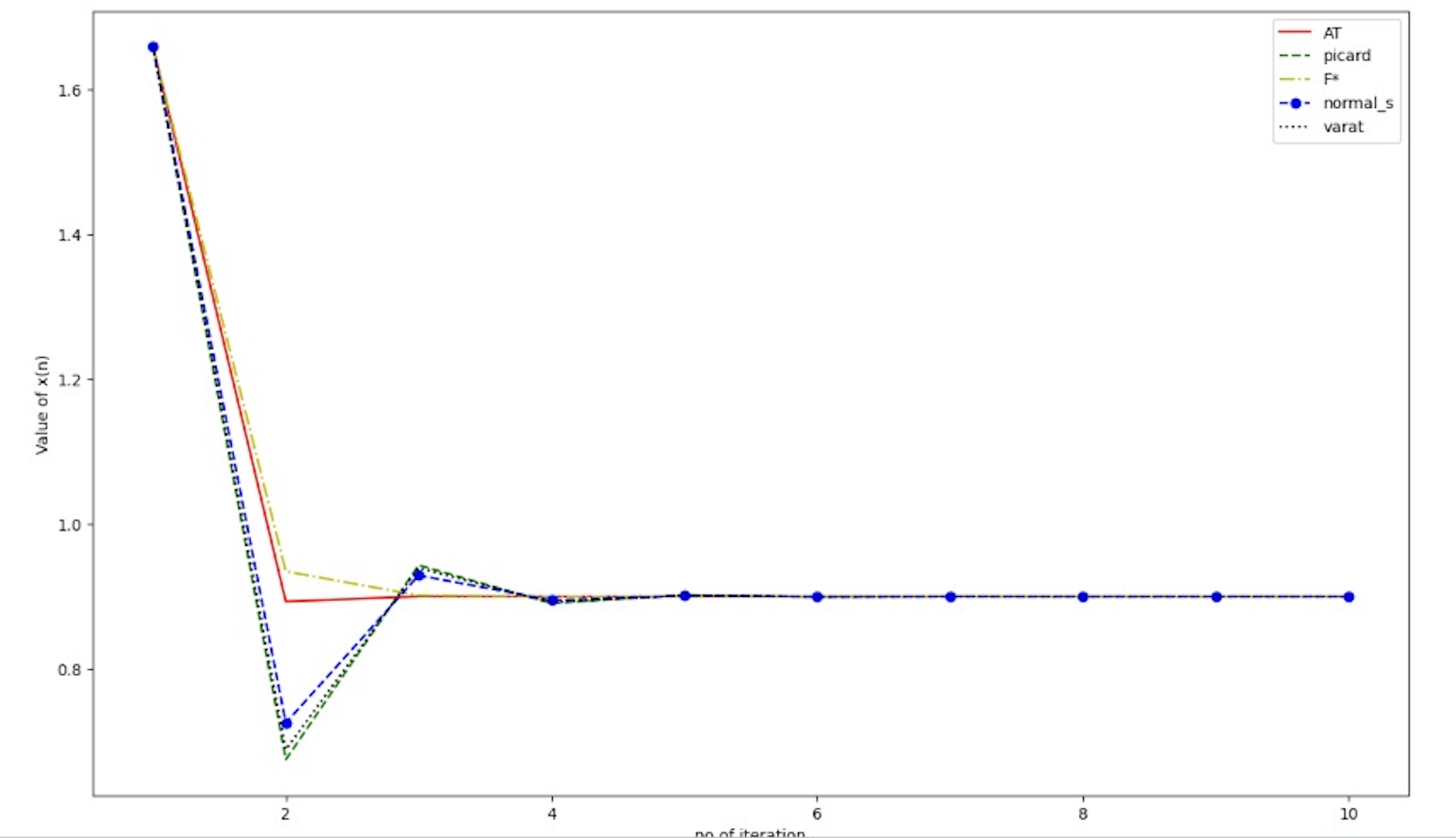}
	\caption{Figure 1. Comparisons of iterations.}
	\label{iterations.jpg}
\end{figure}

\begin{figure}[ht!]
	\centering
	\includegraphics[width=0.8\textwidth]{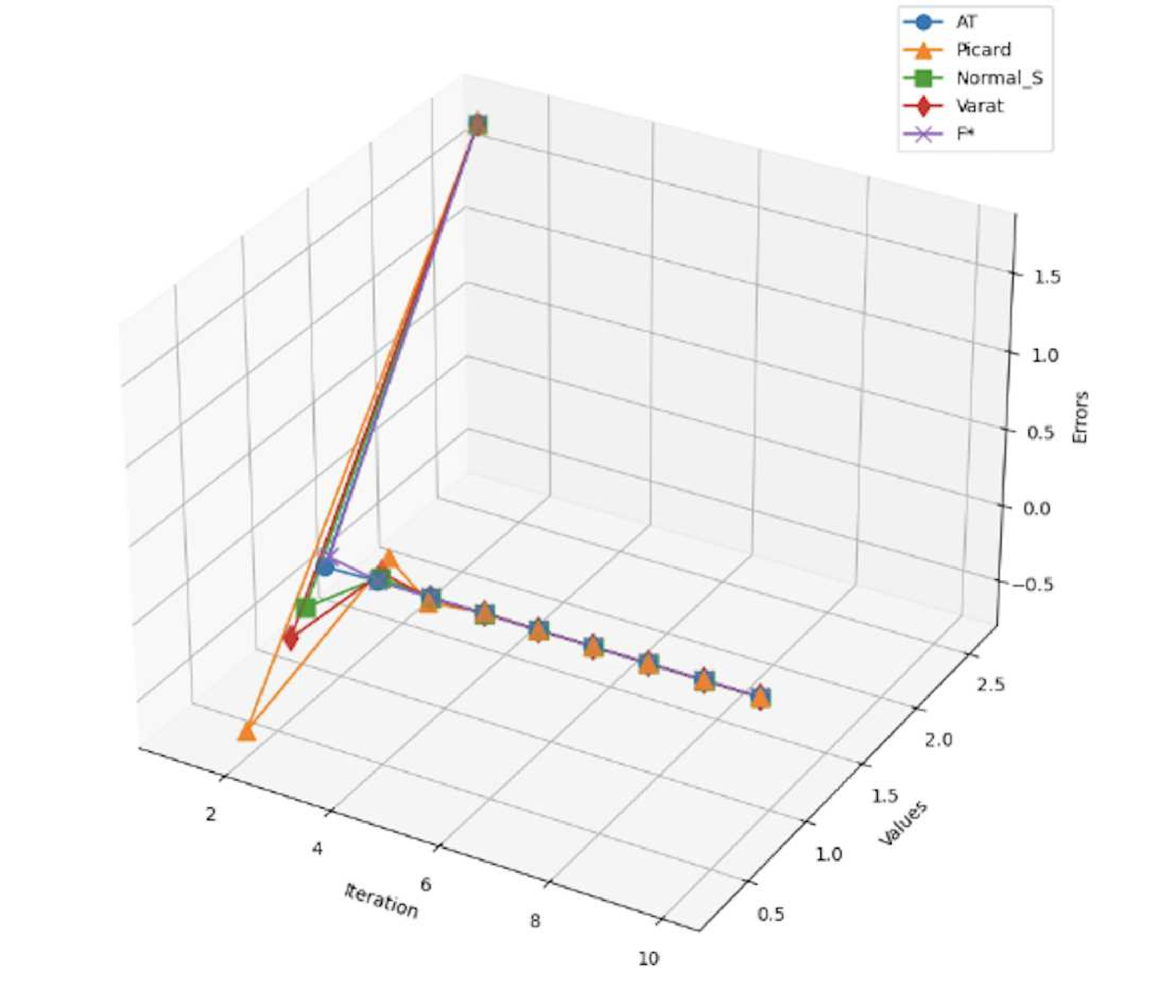}
	\caption{Comparisons  errors of different iterations with  $AT$ iteration}
	\label{graph3derror}
\end{figure}

Using Python ,  it was determined that the AT iterative algorithm described by  converges more rapidly to the fixed point \(s = 0.9\) in comparison to the Mann, S, Picard,  normal-S, and Varat, \(F^{*}\) iterative algorithms. Refer to Table 1 and Fig. 1 for details.

\begin{example}
Let $Q = \mathbb{R}$ set of real numbers with  usual norm, $P = [0,1] \subset Q$. Let $R: P \to P$ be defined as ,\\
$$
R(p) = \begin{cases}
	\frac{p}{2}, & 0 \leq p < 1 \\
	\frac{1}{4}, & \text{when } p = 1
\end{cases}
$$
$R$ with  weak contraction and has fixed point $0$ which is unique but not a contraction mapping.
\end{example}

\section{data dependence}
\begin{theorem}
Let $F$ be an approximate operator of a weak contraction $R$ satisfying \eqref{eq:equation11} and
$\{s_m\}$ be a sequence defined by the $AT$ iterative algorithm \eqref{eq:eq8} for $R$. Now, define a sequence
$\{v_m\}$ for $F$ as follows:
\begin{align}
	\begin{cases}
		v_0 &= v \in P, \\
		v_{m+1} &= F((1 - a_m)g_m + a_m F g_m), \\
		g_m &= \frac{1}{2}[F^2((1 - a_m)v_m + a_mFv_m) + F^2v_m],
	\end{cases}
	\label{eq:14}
\end{align}

where $a_{m}$ is a sequence in $(0, 1)$ satisfying $\frac{1}{2} \leq a_{m}$ for all $a_m \in \mathbb{Z}^+$ and $\sum_{m=0}^{\infty} a_m = \infty$. If $Rs = s$ and $Ft = t$ such that $v_m \to t$ as $m \to \infty$, then we have:
\[
\|s - t\| \leq \frac{5 \epsilon +2\zeta\epsilon+\zeta^2\epsilon}{1-\zeta},
\]
where $\epsilon > 0$ is a fixed number.
\end{theorem}
\begin{proof}
	from equation \eqref{eq:eq8},\eqref{eq:equation11} and \eqref{eq:14} 
	we have ,
	
	\begin{equation}
		\begin{aligned}
			\left\|b_{m}-g_{m}\right\| 
			= &\left\| \frac{1}{2} R^2 s_m + \frac{1}{2} \left( R^2 \left( (1 - a_m) s_m + a_m R s_m \right) \right) - \frac{1}{2} F^2 v_m - \frac{1}{2} \left( F^2 \left( (1 - a_m) v_m + a_m F v_m \right) \right) \right\|\\
			&\leq \frac{1}{2} \| R^2 s_m - F^2 v_m \| + \frac{1}{2} \| R^2 \left( (1 - a_m) s_m + a_m R s_m \right) - F^2 \left( (1 - a_m) v_m + a_m F v_m \right) \|\\
			&\leq \frac{1}{2} \| R^2 s_m - R(Fv_m) \| 
			+ \frac{1}{2} \| R^2 \left( (1 - a_m) s_m + a_m R s_m \right) - R(F( \left( (1 - a_m) v_m + a_m F v_m \right))) \| \\
			& +\frac{1}{2} \| F^2 v_m - R(F v_m) \| + \frac{1}{2} \| R(F( \left( (1 - a_m) v_m + a_m F v_m \right))) - F^2 \left( (1 - a_m) v_m + a_m F v_m \right) \|\\
			&\leq \frac{\zeta}{2} \| R s_m - F v_m \| 
			+ \frac{\zeta}{2} \| R \left( (1 - a_m) s_m + a_m R s_m \right) - F \left( (1 - a_m) v_m + a_m F v_m \right) \| + \epsilon\\
			& +\frac{L}{2}\|Rs_m - R^2 s_m\|+\frac{L}{2}\|R \left( (1 - a_m) s_m + a_m R s_m \right)-R^2 \left( (1 - a_m) s_m + a_m R s_m \right)\|\\
			& \leq\frac{\zeta^2}{2}\|s_{m}-v_{m}\|+\frac{\zeta L}{2}\|s_{m}-Rs_{m}\|+\frac{\zeta \epsilon}{2}+\frac{\zeta \epsilon}{2}+\\
	        &\frac{\zeta^2}{2}\|(1 - a_m) s_m +
			 a_m R s_m - (1 - a_m) v_m + a_m F v_m\|\\
			&+\frac{\zeta L}{2}\|R((1 - a_m) s_m + a_m R s_m)-R^2((1 - a_m) s_m + a_m R s_m))\|
			+\epsilon +\frac{\zeta L}{2}\|s_{m}-R s_m\|\\
			&+\frac{\zeta L}{2}\|s_{m}-R s_m\|
			\frac{\zeta L}{2}\|((1 - a_m) s_m + a_m R s_m)-R((1 - a_m) s_m + a_m R s_m))\|+\\
			&\frac{L^2}{2}\|((1 - a_m) s_m + a_m R s_m)-R((1 - a_m) s_m + a_m R s_m))\|\\
			&\leq \|s_{m}-v_{m}\| +\frac{\zeta^2 L+2\zeta L+L^2}{2} \|s_{m}-Rs_{m}\|+\frac{\zeta^2 \epsilon}{2}+\zeta \epsilon+\epsilon+\frac{2\zeta L +L^2}{2}(1+\zeta)\|s_{m}-s\|
			\label{eq:15}
		\end{aligned}
	\end{equation}
	from equation \eqref{eq:15}
	
	\begin{equation}
		\begin{aligned}
			\left\|s_{m+1}-v_{m+1}\right\|= & \left\|R\left((1-a_{m}) b_{m}+a_{m} R b_{m}\right)-F\left((1-a_{m}) g_{m}+a_{m} F g_{m}\right)\right\| \\
			\leq & \left\|R\left((1-a_{m}) b_{m}+a_{m} R b_{m}\right)-R\left((1-a_{m}) g_{m}+a_{m} F g_{m}\right)\right\| \\
			& +\left\|R\left((1-a_{m}) g_{m}+a_{m} F g_{m}\right)-F\left((1-a_{m}) g_{m}+a_{m} F g_{m}\right)\right\| \\
			\leq & \zeta\left((1-a_{m})\left\|b_{m}-g_{m}\right\|+a_{m}\left\|R b_{m}-F g_{m}\right\|\right) \\
			& +L\left\|\left(1-a_{m}\right) b_{m}+a_{m} R b_{m}-R\left((1-a_{m}) b_{m}+a_{m} R b_{m}\right)\right\|+\epsilon \\
			\leq & \zeta\left((1-a_{m})\left\|b_{m}-g_{m}\right\|+a_{m}(\zeta\left\| b_{m}- g_{m}\right\|+L\|b_{m}-R_{m}\|+\epsilon\right) \\
			& +L\left\|\left(1-a_{m}\right) b_{m}+a_{m}\ R b_{m}-R\left((1-a_{m}) b_{m}+a_{m} R b_{m}\right)\right\|+\epsilon \\
			\leq & \zeta(\left((1-(1-\zeta)a_{m})\left\|b_{m}-g_{m}\right\|+a_{m} L\left\|b_{m}-R b_{m}\right\|+a_{m}\epsilon\right) \\
			& +L(1- a_{m})\|b_{m}-s\|+La_{m}\zeta\|b_{m}-s\|+L(1- a_{m})\zeta\|b_{m}-s\|+\\
			 &La_{m}\zeta^2\|b_{m}-s\|+\epsilon 
			\label{eq:16}
		\end{aligned}
	\end{equation}
	
	Since $\zeta \in (0, 1)$, $a_m \in (0, 1)$ with $a_m \geq \frac{1}{2}$; therefore, using the inequalities $\zeta < 1$, $\zeta^2 < 1$, $1 - a_{m} \leq a_{m}$, and $1 - (1 - \zeta)a_{m} \leq 1$ in \eqref{eq:16}, and from equation \eqref{eq:15} we get:
	\begin{align*}
		\|s_{m+1}-v_{m+1}\| \leq & (1-(1-\zeta)a_{m})( \|s_{m}-v_{m}\|) +(\zeta^2 L+2\zeta L+L^2)a_m \|s_{m}-Rs_{m}\|+\\
		&(\zeta^2 \epsilon)a_m+2\zeta \epsilon a_{m}+4\epsilon a_{m}
		 +(2\zeta L +L^2)a_m (1+\zeta)\|s_{m}-s\|
		+a_{m} L\left\|b_{m}-R b_{m}\right\|\\
		&+a_{m}\epsilon +2La_{m}\|b_{m}-s\|+La_{m}\zeta^2\|b_{m}-s\| \\
	\end{align*}
	Now, define:
	
	$$
	\begin{aligned}
		p_{m} & =:\left\|s_{m}-v_{m}\right\|, \\
		q_{m} & =: a_{m}(1-\zeta) \in(0,1), \\
		\delta_{m}& = \frac{\left( \zeta^2 L+2\zeta L+L^2) \|s_{m}-Rs_{m}\| +(2\zeta L +L^2) (1+\zeta)\|s_{m}-s\| +L\|b_{m}-R b_{m}\| \right)}{1-\zeta}\\
		& + \frac{ 2L\|b_{m}-s\| + L\zeta^2\|b_{m}-s\|+5\epsilon +\zeta^2 \epsilon + 2\zeta\epsilon}{1-\zeta}
	\end{aligned}
	$$

	$$
	p_{m+1} \leq\left(1-q_{m}\right) p_{m}+\delta_{m} q_{m} .
	$$
	
	All the conditions of Lemma 1.8 are satisfied. Hence, applying Lemma 1.8, we get:
	
	$$
	\begin{aligned}
		& 0 \leq \lim \sup _{m \rightarrow \infty}\left\|s_{m}-v_{m}\right\| \\
		& 0 \leq \lim \sup _{m \rightarrow \infty} \frac{\left( \zeta^2 L+2\zeta L+L^2) \|s_{m}-Rs_{m}\| +(2\zeta L +L^2) (1+\zeta)\|s_{m}-s\| +L\|b_{m}-R b_{m}\| \right)}{1-\zeta}\\
		& \frac{ 2L\|b_{m}-s\| + L\zeta^2\|b_{m}-s\|+5\epsilon +\zeta^2 \epsilon + 2\zeta\epsilon}{1-\zeta} \\
		& =\frac{5 \epsilon +2\zeta\epsilon+\zeta^2\epsilon}{1-\zeta} .
	\end{aligned}
	$$
	
	In view of Theorem 2.1, we know that $s_{m} \rightarrow s$, and using hypothesis, we obtain:
	
	$$
	\|s-t\| \leq \frac{5 \epsilon +2\zeta\epsilon+\zeta^2\epsilon}{1-\zeta}
	$$
\end{proof}
\section{Example}
\begin{example}
Consider
$$
R(t) = \cos\left(\frac{t}{2}\right) \quad \text{for } t \in [0,1]
$$\\
clearly R is a weak contraction.

Let \( F(t) = 1 - (0.25)t^2 + (0.0026)t^4 \), \( t \in [0,1] \)  and  $a_{m} =0.5$.

$$ \max_{t\in [0,1]} \lvert R(t)-F(t) \rvert   = 0.124978  $$

hence $\epsilon =0.124978$ and R has fixed point  $ 0.9$\\

\begin{table}[h]\label{table3}
	\centering
	\begin{tabular}{@{}ll@{}}
		It no & Iter. algorithm   \\ 
		1     &  1.658950           \\
		2     &  0.893291           \\
		3     &  0.900422           \\
		4     &  0.900367           \\
		5     &  0.900367           \\
		6     &  0.900367             \\
		7     &  0.900367             \\
		8     &  0.900367             \\
	\end{tabular}
\end{table}
\begin{align}
	\begin{cases}
		v_0 &= v \in P, \\
		v_{m+1} &= F((1 - 0.5)g_m + 0.5 F g_m), \\
		g_m &= \frac{1}{2}[F^2((1 - 0.5)v_m + 0.5 Fv_m) + F^2v_m],\\
	\end{cases}
	\end {align}

	Clearly, ${v_{m}}$ converges to $q=0.900367$ fixed point of F\\
	
	By Theorem 3.1 we  have
	
	$$\|p-q\| \leq \frac{5 \epsilon +2\zeta\epsilon+\zeta^2\epsilon}{1-\zeta}$$
	
	we have $ L =0$  and $\zeta=\frac{1}{2}$ thus we get
	
	$$\|p-q\| \leq \frac{5 \epsilon +2\zeta\epsilon+\zeta^2\epsilon}{1-\zeta}  $$
	as $\epsilon$ = 0.124978\;
	,$\zeta=\frac{1}{2}$\; we\;have

	$$\|p-q\| \leq 1.5625  $$
	
	Thus from the theorem we have 
	$\|p-q\| \leq 1.5625 $
	and we have actually
	
	$\|p-q\| = 0.000367  $.
	\end{example}

	\section{Application}
	We explore a unique solution for the given initial value problem through the application of the $AT$ algorithm.
	-\\
	\begin{equation}
		\left\{
		\begin{array}{l}
			y^{n}(t) + f\left(t, y(t), y^{1}(t), \ldots, y^{n-1}(t)\right) = 0 \\
			y\left(t_{0}\right) = p_{0}, \quad y^{\prime}(t_{0}) = p_{1}, \ldots, y^{n-1}\left(t_{0}\right) = p_{n-1}
		\end{array}
		\right\} \label{17}
	\end{equation}

	\[
	\text{where } p_i \in \mathbb{R} \quad \forall i = 0, 1, \ldots, n-1
	\]

	$f:[a, b] \times \mathbb{R}^{n} $ cnts function satisfies following condition such that

	\begin{equation}
		\begin{aligned}
			| f\left(t, y, y^{\prime},\ldots,y^{n-1}\right) - f\left(t, z, z^{\prime},\ldots,z^{n-1}\right)| & \leq a_{1}|y-z|  + a_{2}\left|y^{\prime}-z^{\prime}\right| + \ldots + a_{n}\left|y^{n-1}-z^{n-1}\right|
		\end{aligned}
		\label{equation:18}
	\end{equation}

	$\forall t \in[a, b]$
	
	$\forall y^{i}, z^{i} \in \mathbb{R} \quad \forall i=1,2-n-1$ and $a_{i} \geq 0$\\
	Let $X=C^{n}[a, b]$
	
	\begin{equation}\label{equation:19}
		\begin{aligned}
			\|x-y\|_{\infty} = \max \left\{\sup_{t \in [a, b]} |x(t) - y(t)|, \sup_{t \in [a, b]}|x'(t) - y'(t)|, \ldots, \sup_{t \in [a, b]}|x^{(n-1)}(t) - y^{(n-1)}(t)|\right\}
		\end{aligned}
	\end{equation}

	$\forall x, y \in X$
	
	Then $(X, \|.\|_{\infty})$ is a complete normed linear  space.
	Define $T$ on $X$ by

	\begin{equation}
		T(y(t)) = \int_{a}^{b} g(t) f\left(s, y(s)-y^{n-1}(s)\right) \, ds \label{equation:20}
	\end{equation}
	
	$$
	\text{where } g(t) \in X.\\
	$$
	\begin{theorem}
	Let $X = C^n[a, b]$ be a complete normed linear  space with norm defined by \ref{equation:19}. Let
	$f:[a, b] \times \mathbb{R}^{n} $ continuous function such that the Lipschitz condition \ref{equation:18} is satisfied. Suppose that
	\begin{equation}
		\alpha = \sum_{i=1}^{n} \alpha_i M < 1. \label{equation:21}
	\end{equation}
	
	$$
	\text{where } \frac{\partial^i g(t)}{\partial^i t} \leq M_i \quad \text{and} \quad M = \max\{M_0, \ldots, M_{n-1}\}
	$$

	Then the problem \ref{17} has a unique solution in $X$.
	\end{theorem}
	\begin{proof}
		
		\[
		\begin{aligned}
			\left| T \left(\frac{d^i}{dt^i} y(t)\right) - T \left( \frac{d^i}{dt^i}z(t) \right) \right|
			&= \left\lvert\int_{a}^{b} \frac{\partial^i g(t)}{\partial^i t} f\left(t, y, y^{\prime},\ldots,y^{n-1}\right) - f\left(t, z, z^{\prime},\ldots,z^{n-1}\right) \, ds\right\rvert \\
			&\leq \int_{a}^{b} \frac{\partial^i g(t)}{\partial^i t} \left| f\left(t, y, y^{\prime},\ldots,y^{n-1}\right) - f\left(t, z, z^{\prime},\ldots z^{n-1}\right)\right| \, ds \\
			&\leq \int_{a}^{b} \frac{\partial^i g(t)}{\partial^i t} \left( a_{1}\|y-z\| + a_{2}\left|y^{\prime}-z^{\prime}\right| \ldots+a_{n}\left|y^{n-1}-z^{n-1}\right| \right) \, ds \\
			&\leq \int_{a}^{b} \frac{\partial^i g(t)}{\partial^i t} \left( a_{1}\|y-z\| + a_{2}\|y-z\| \ldots +a_{n}\|y-z\| \right) \, ds \\
			&\leq \sum_{i=1}^{n} \alpha_i M_i \|y-z\|
		\end{aligned}
		\]
		
		\begin{equation}
			\begin{aligned}
				\|T(y)- T(z)\| &= \max \left\{ \sup_{t \in [a, b]} \lvert T(y(t)) - T(z(t)) \rvert, \ldots, \sup_{t \in [a, b]} \lvert T(y^{(n-1)}(t)) - T(z^{(n-1)}(t)) \rvert \right\}\\
				\leq \alpha \|y-z\|
			\end{aligned}
		\end{equation}
		We see that the mapping $T$ defined by \ref{equation:20} is a contraction and hence a weak contraction. So, using $AT$ iteration  $T$ has a unique fixed point in $X$. Therefore, the problem \eqref{17} has a unique solution in $X$.
			\end{proof}

	\section{Conclussion}
	We propose a  two-step iterative algorithm designed to approximate fixed points of weak contractions within a complete normed linear space. Our algorithm demonstrates enhanced efficiency and faster convergence compared to several established iterative methods, supported by the findings detailed in Theorem 2.3. Theorem 2.2 establishes the near R-stability of the $AT$ iterative algorithm . Additionally, by the $AT$ algorithm, we derived a data dependence result and provided an illustrative example to validate its credibility.\\
	We have two open questions that we can extend \\
	\textbf{Question 1} Is it possible to formulate a novel two-step iterative method that exhibits a faster convergence rate than the AT iterative algorithm?\\
	\textbf{Question 2} Does the sequence {sm} defined by the AT iterative algorithm converge towards a fixed point of either a contractive-like operator or a non-expansive operator?



\end{document}